\documentclass[preprint,10pt]{elsarticle}
\usepackage{amssymb}
\usepackage{amsthm}
\usepackage{amsmath}
\usepackage{epic}
\newtheorem{thm}{Theorem}[section]
\newtheorem{cor}[thm]{Corollary}
\newtheorem{lem}[thm]{Lemma}
\newtheorem{prop}[thm]{Proposition}

\journal{}

\begin{document}

\begin{frontmatter}

\title{Bounds on the $\alpha$-distance spectrum of graphs}
\author{Yang Yang\fnref{label1}}
\author{Lizhu Sun\fnref{label2}}\ead{sunlizhu678876@126.com}
\author{Changjiang Bu$^{a,b}$} \ead{buchangjiang@hrbeu.edu.cn}

\address[label1]{College of Automation,, Harbin Engineering University, Harbin 150001, PR China}
\address[label2]{College of  Mathematical Sciences, Harbin Engineering University, Harbin 150001, PR China}

\begin{abstract}
For a simple, undirected and connected graph $G$, $D_{\alpha}(G) = \alpha Tr(G) + (1-\alpha) D(G)$ is called the $\alpha$-distance matrix of $G$, where $\alpha\in [0,1]$, $D(G)$ is the distance matrix of $G$, and $Tr(G)$ is the vertex transmission diagonal matrix of $G$. Recently, the $\alpha$-distance energy of $G$ was defined based on the spectra of $D_{\alpha}(G)$. In this paper, we define the $\alpha$-distance Estrada index of $G$ in terms of the eigenvalues of $D_{\alpha}(G)$.  And we give some bounds on the spectral radius of $D_{\alpha}(G)$, $\alpha$-distance energy and $\alpha$-distance Estrada index of $G$.
\end{abstract}

\begin{keyword}
Distance matrix, Energy of graph, Estrada index, $\alpha$-distance spectra
\MSC[2010]
05C50
\end{keyword}

\end{frontmatter}

\section{Introduction}

\subsection{Distance spectrum of graphs}

In this paper, we consider the simple, undirected and connected graphs. Let $G=(V(G)$, $E(G))$ be a graph with the vertex set $V(G) = \{v_1,\ldots,v_n\}$ and edge set
$E(G)$.
The distance between two vertices $v_i,v_j\in V(G)$ is the length of the shortest path between $v_i$ and $v_j$, denoted by $d(v_i,v_j)$.
The \textit{Wiener index} $W(G)$ of the graph $G$ is the sum of the distances between all pairs of vertices in $G$, i.e.,
$W(G)=\frac{1}{2}\sum_{v_i,v_j\in V(G), i\neq j}d(v_i,v_j).$
The matrix $D(G)=(d_{i,j})\in \mathbb{R}^{n\times n}$ with entries $d_{i,j}=d(v_i,v_j)$ is called the \textit{distance matrix} of  $G$, where $i,j\in \{1,2,\cdots,n\}$. The spectra of $D(G)$ is called the \textit{distance spectra} of $G$.

The study of the distance spectra arises from the work of R. Graham and H.O. Pollack in 1971 (see [11]). In [11], the determinant of the distance matrix of a tree $T$ was given as  $\det(D(T))=(-1)^{n-1}(n-1)2^{n-2}$, where $n=|V(T)|$. And Graham and Pollack proved $D(T)$ has $1$ positive eigenvalue and $n-1$ negative eigenvalues. In 1978, Graham and Lov\'{a}sz established the representation for the inverse of $D(T)$ (see [10]). The distance spectrum of graphs has attracted much attention [27].

\subsection{$\alpha$-spectrum of graphs}

The \textit{adjacency matrix} of the graph $G$ is $\widetilde{A}(G)=(\widetilde{a}_{ij})\in \mathbb{R}^{n\times n}$, where $\widetilde{a}_{ij}=1$ if $(i,j)\in E(G)$, and $\widetilde{a}_{ij}=0$ otherwise.
The \textit{Laplacian matrix} and \textit{signless Laplacian matrix} of $G$ are
$$\widetilde{L}(G)=\widetilde{D}(G)-\widetilde{A}(G) \mbox{ and } \widetilde{Q}(G)=\widetilde{D}(G)+\widetilde{A}(G),$$
respectively, where $\widetilde{D}(G)= \mathrm{diag}(\widetilde{d}_{v_1},\cdots,\widetilde{d}_{v_n})\in \mathbb{R}^{n\times n}$ and $\widetilde{d}_{v_i}$ is the degree of $v_i$, $i=1,2,\cdots,n$.

It is well-known that the spectrum of the adjacency matrix, Laplacian matrix and signless Laplacian matrix of a graph were widely investigated [5].
In 2013, the study of the spectrum of Laplacian matrix and signless Laplacian matrix was extended to distance Laplacian matrices and distance signless Laplacian matrices defined as in Equation (1) (see [1]). In 2016,  the study of the spectrum of Laplacian matrix and signless Laplacian matrix was generalized to a convex combination of $\widetilde{D}(G)$ and $\widetilde{A}(G)$ defined as $ A_\alpha{(G)} = \alpha \widetilde{D}(G) + (1-\alpha)\widetilde{A}(G),~\alpha \in [0,1]$ (see [24]). Recently, the above study was further extended to the $\alpha$-distance matrices defined as in Equation (2) (see [3]).

For a vertex $v_i\in V(G)$, the sum of the distances between $v_i$ and all the other vertices in $V(G)$ is called the \textit{transimission}  of $v_i$, denoted by $Tr(v_i)$, that is $$Tr(v_i)=\sum_{v_j\in V(G),i\neq j} d(v_i,v_j).$$
A graph $G$ is said to be \textit{transmission regular} if the transimissions of all the vertices in  $V(G)$ are equal.
In 2013, Aouchiche and Hansen [1] defined the distance Laplacian matrix $\mathcal{L}(G)$ and distance signless Laplacian matrix $\mathcal{Q} (G)$ of the graph $G$,
\begin{align}
\mathcal{L}(G)=Tr(G)-D(G) \mbox{ and } \mathcal{Q} (G) = Tr(G) + D(G),
\end{align}
where $Tr(G)=\mathrm{diag}(Tr(v_1),\cdots,Tr(v_n))$.

For a transmission regular graph $G$, the characteristic polynomials of $\mathcal{L}(G)$ and $\mathcal{Q} (G)$ were calculated in [1]. In [21,28,29], the bounds on the spectral radii for the distance signless Laplacian matrix of trees, unicyclic graphs and bicyclic graphs were characterized, respectively.

In [3], the \textit{$\alpha$-distance matrix} of a graph $G$
\begin{align}
D_{\alpha}(G) = \alpha Tr(G) + (1-\alpha) D(G),~\alpha \in[0,1],
\end{align}
was defined. Clearly, $D_0(G)=D(G)$, $2D_{\frac{1}{2}}(G) = \mathcal{Q}(G)$ and $D_1(G) = Tr(G)$.
The spectra of $D_\alpha(G)$ is called the \textit{$\alpha$-distance spectra} of $G$.
Since $D_{\alpha}(G)$ is a real symmetric matrix, the eigenvalues of $D_{\alpha}(G)$ are real. Let $\sigma_1(G)\geq \sigma_2(G) \geq \cdots \geq\sigma_n(G)$ be the eigenvalues of $D_{\alpha}(G)$. And let $\rho_{\alpha}(G)= \max \{\sigma_i(G) | i=1, 2, \cdots, n\}$ be the \textit{spectral radius} of $D_{\alpha}(G)$. We know that $D_{\alpha}(G)$ is a nonnegative weakly irreducible matrix. From the Perron-Frobenius Theorem, we have $\sigma_1(G)=\rho_{\alpha}(G)$. And $D_{\alpha}(G)$  is positive semidefiniteness for $\alpha \in [\frac{1}{2},1]$.

In [3],  upper and lower bounds for the spectral radius of the $\alpha$-distance matrix were established.
In [4], authors characterized the unique graph with minimum spectral radius of the $\alpha$-distance matrix among the connected graphs with fixed chromatic number.
In [20], a lower bound on the $k$-th smallest eigenvalue of the $\alpha$-distance matrix was given.
In [14],  the bounds on the spectral radius of the $\alpha$-distance matrix were established.

\subsection{$\alpha$-distance energy of graphs}

Let $\lambda_i(G)$, $\kappa_i(G)$ and $\nu_i(G)$ be the eigenvalues of $\widetilde{A}(G)$, $D(G)$ and $\mathcal{Q} (G)$
respectively, where $i=1,2\cdots,n$ and $n=|V(G)|$. The energy of $G$ is $E_\pi(G)=\sum_{i=1}^n|\lambda_i(G)|$ (see [12,13]).
The sum $DE(G) = \sum_{i=1}^n |\kappa_i(G)|$ is called the distance energy of $G$ (see [17]),
and $DSLE(G) = \sum_{i=1}^n|\nu_i(G) - \frac{2W(G)}{n}|$ is called the distance signless Laplacian energy [6].

Graph energy has important applications in the fields of mathematics and chemistry. There are many researches on the above kinds of graph energy.
Scholars gave the bounds on the energy of graphs, for example the McClelland$^\prime s$ bounds [23], Koolen-Moulton$^\prime s$ bounds [18] and so on [2]. In [17], the distance energy of some graphs were calculated.

Recently, Zhou extended the concept of graph energy to a more general form called \textit{$\alpha$-distance energy}
$$\varsigma_{\alpha}(G) =\sum_{i=1}^{n}|\sigma_i(G) -\frac{2\alpha W(G)}{n}|,~\alpha \in[0,1],$$
where $\sigma_i(G)$ is the eigenvalue of $D_{\alpha}(G)$, $i=1,2,\cdots,n$, $n=|V(G)|$ (see [14]). Clearly, $\varsigma_0(G)=DE(G)$ and $\varsigma_{1/2}(G)=\frac{1}{2}DSLE(G)$.

\subsection{Main work}
In this paper, we give some bounds on the $\alpha$-distance energy of graphs and stars in terms of the parameter $\alpha$ and the vertex number. And we establish some bounds for the spectral radius of the $\alpha$-distance matrix by using the transimission of vertices and Wiener index. Further, we define the $\alpha$-distance Estrada index of a graph, and obtain some bounds on the $\alpha$-distance Estrada index.

\section{Some bounds for the $\alpha$-distance energy of graphs}
To begin with this section, we introduce some notations.
The \textit{average transmission} of the graph $G$, denoted by $t(G)$, is defined by
$$t(G)=\frac{1}{n}\sum_{i=1}^{n}Tr(v_i),$$
where $n=|V(G)|$. Clearly,
$t(G)=\frac{2W(G)}{n}.$
Let $S= \sum_{1\leq i<j\leq n} d^2(v_i,v_j)$ (see [3]).

\begin{lem}
\textup{[3]}
Let $G$ be a  graph with $n$ vertices. Then
\begin{eqnarray*}
\rho_{\alpha}(G) \geq \frac{2W(G)}{n},
\end{eqnarray*}
the equality holds if and only if $G$ is a transmission regular graph.
\end{lem}

For a graph $G$ and $v\in V(G)$, let $N_G(v)=|\{u:(v,u)\in E(G)\}|$.

\begin{lem} \textup{[3]}
Let $G$ be a graph with $n$ vertices. Let $S$ be a subset of $V(G)$ such that $N_G(x) = N_G(y)$ for any $x,y \in S$.\\
(1) If $S$ is an independent set, then $Tr(v)$ is a constant $h$ for each $v \in S$, and $D_\alpha(G)$ has $\alpha(h+2)-2$ as an eigenvalue with multiplicity at least $|S|-1$.\\
(2) If $S$ is a clique, then $Tr(v)$ is a constant $h^\ast$ for each $v \in S$, and $D_\alpha(G)$ has $\alpha(h^\ast+1)-1$ as an eigenvalue with multiplicity at least $|S|-1$.
\end{lem}

From Lemma 2.2, we have the following result.
\begin{prop}
For a graph $G$, $D_\alpha(G)$ ($\alpha\in [0,1)$) has two distinct eigenvalues if and only if $G$ is a complete graph.
\end{prop}

Let $S_{n}$ be a star with $n$ vertices.

\begin{lem}\textup{[4,14,20]}
Let $T$ be a tree with $n\geqslant 4$ vertices. Then
\begin{eqnarray*}
\rho_\alpha(T)\geqslant\rho_\alpha(S_n) = \frac{(\alpha + 2)n-4 + \sqrt{[ (\alpha + 2)n -4]^2 + 4[(n-1)(2\alpha-2n\alpha + 1)]}}{2},
\end{eqnarray*}
the equality holds if and only if $G \cong S_n$.
\end{lem}

\begin{prop}
The $\alpha$-distance spectra of $S_n$ consists of\\
(1) $(2n-1)\alpha-2$ with multiplicity $n-2$;\\
(2) $\frac{\alpha n + 2n -4 \pm \sqrt{(\alpha -2)^2n^2 + 8\alpha n -12n-8\alpha +12}}{2}$.
\end{prop}

\begin{proof}
Since $S_{n}$ satisfies the statement (1) of Lemma 2.2, and the number of the vertices in the independent set is $n-1$ and the transimission of each vertex in the independent set is $2n-3$, we have  $(2n-1)\alpha-2$ is an eigenvalue of $D_\alpha(S_{n})$ with multiplicity at least $n-2$. From Lemma 2.4, the statement (2) holds.
\end{proof}

Let $K_n$ be a complete graph with $n$ vertices.

\begin{lem}\textup{[9]}
Let $G$ be a graph with $n$ vertices. Then
\begin{eqnarray*}
W(G)\geq \frac{n(n-1)}{2},
\end{eqnarray*}
the equality holds if and only if $G \cong K_n$.
\end{lem}

Next, we give some bounds for the $\alpha$-distance energy of a graph and characterize the $\alpha$-distance energy of $S_n$ by using the parameter $\alpha$ and the vertex number.

\begin{thm}
Let $G$ be a connected graph with $n$ vertices. Then
\begin{align*}
\varsigma_{\alpha}(G) \geq 2(1-\alpha)(n-1),~\alpha \in [1/2,1),
\end{align*}
the equality holds if and only if $G\cong K_n$.
\end{thm}

\begin{proof}
Let $\sigma_1(G)\geq \sigma_2(G)\geq \cdots\geq \sigma_n(G)$ be the eigenvalues of $D_\alpha(G)$.
It is easy to see that
\begin{align}
\sum_{i=1}^n \sigma_i(G) = {\rm Trace}(D_\alpha(G))=\sum_{i=1}^n \alpha Tr(v_i) =n \alpha t(G) = 2\alpha W(G).
\end{align}
Lemma 2.1 gives $\sigma_1(G)\geq t(G)\geq \alpha t(G)$. Suppose that $ \iota $ is the largest number such that $\sigma_\iota(G)\geq  \alpha t(G)$. And it follows from Equation (3) that
\begin{align*}
\varsigma_{\alpha}(G)&=\sum_{i=1}^\iota (\sigma_i(G) - \alpha t(G)) + \sum_{i=\iota+1}^n ( \alpha t(G)-\sigma_i (G))\\
                    &= \sum_{i=1}^\iota \sigma_i (G) - \iota \alpha t(G) + (n - \iota) \alpha t(G) - \sum_{i=\iota+1}^n \sigma_i(G)\\
                    &=\sum_{i=1}^\iota \sigma_i (G) - \iota \alpha t(G) +\sum_{i=1}^n \sigma_i(G)- \iota \alpha t(G) - \sum_{i=\iota+1}^n \sigma_i(G)\\
                    &= 2\sum_{i=1}^{\iota}(\sigma_i (G) - \alpha t(G))\\
                    &\geq  2( \sigma_1(G)- \alpha t(G)).
\end{align*}
From Lemmas 2.1 and 2.6, we have
\begin{align*}
2( \sigma_1(G)- \alpha t(G)) \geq 2(\frac{2W(G)}{n}- \alpha\frac{2W(G)}{n})= 4(1-\alpha)\frac{W(G)}{n}\geq 2(1-\alpha)(n-1),
\end{align*}
and $\varsigma_{\alpha}(G)=2(1-\alpha)(n-1)$ if and only if $G\cong K_n$.
\end{proof}

\begin{thm}
The $\alpha$-distance energy of $S_{n}$ is
\begin{align*}
&\varsigma_{\alpha}(S_{n}) =\\
& \frac{\alpha (-3n^2+8n-4) + 2n^2 -4n + n\sqrt{(\alpha -2)^2n^2 + 8\alpha n -12n-8\alpha +12}}{2n} \\
&+|\frac{\alpha (-3n^2+8n-4) + 2n^2 -4n - n\sqrt{(\alpha -2)^2n^2 + 8\alpha n -12n-8\alpha +12}}{2n}|\\
&+ |3\alpha - 2-\frac{2\alpha}{n}|(n-2).
\end{align*}
\end{thm}

\begin{proof}
We know that 
$
t(S_{n})=2n-4+\frac{2}{n}.
$
And from Proposition 2.5, we obtain
\begin{align*}
\varsigma_{\alpha}(S_{n})& =\sum_{i=1}^{n}|\sigma_i(G) -\alpha t(S_{n})|\\
& =\frac{\alpha (-3n^2+8n-4) + 2n^2 -4n + n\sqrt{(\alpha -2)^2n^2 + 8\alpha n -12n-8\alpha +12}}{2n} \\
&+|\frac{\alpha (-3n^2+8n-4) + 2n^2 -4n - n\sqrt{(\alpha -2)^2n^2 + 8\alpha n -12n-8\alpha +12}}{2n}|\\
&+ |3\alpha - 2-\frac{2\alpha}{n}|(n-2).
\end{align*}
\end{proof}

Let $\| M \| _F$ denote the Frobenius norm of a matrix $M=( m_{i,j} )\in \mathbb{C}^{n\times n}$, that is
$$
\| M \| _F = \sqrt{\sum_{i=1}^n \sum_{j=1}^n |m_{i,j}|^2 }.
$$
We know that  $\| M \| _F^2 = \sum_{i=1}^n|\tilde{\lambda}_i(M)|^2$, where  $\tilde{\lambda}_1(M),\ldots,\tilde{\lambda}_n(M)$ are the eigenvalues of $M$.
For a graph $G$ with $n$ vertices,
\begin{align*}
\sum_{i=1}^n \xi_i^2(G) = 2|\sum_{i=1}^{n-1}\sum_{j=i+1}^n \xi_i(G)\xi_j(G)|=\|D(G)\|_F^2,
\end{align*}
where $\xi_1(G),\xi_2(G),\ldots,\xi_n(G)$ are the eigenvalues of $D(G)$ (see [7,8]).
Similarly to the above equation, we have
\begin{align}
\sum_{i=1}^n \sigma_i^2(G) = \|D_{\alpha}(G)\|_F^2 = \alpha ^2 \sum_{i=1}^n Tr^2(v_i) + (1-\alpha)^2\|D(G)\|^2_F,
\end{align}
where $\sigma_1(G),\sigma_2(G),\ldots,\sigma_n(G)$ are the eigenvalues of $D_\alpha(G)$.

\begin{lem}
\textup{[8]}
Let $G$ be a graph with $n$ vertices. Then
\begin{eqnarray*}
\|D(G) \|_F^2 < \frac{1}{n} (\sum_{i=1}^n Tr(v_i) )^2.
\end{eqnarray*}
\end{lem}

From Lemma 2.9, the following result can be obtained directly.
\begin{prop}
Let $G$ be a connected graph with $n$ vertices. Then 
\begin{align*}
\|D_{\alpha}(G) \|_F < \sqrt{\alpha ^2 \sum_{i=1}^n Tr^2(v_i) +  (1-\alpha)^2 \frac{1}{n} (\sum_{i=1}^n Tr(v_i) )^2}.
\end{align*}
\end{prop}

In the following, we establish some bounds for the $\alpha$-distance energy of a graph $G$ by using the Frobenius norm of $D(G)$, transmissions of vertices and $ W(G)$.

\begin{thm}
Let $G$ be a graph with $n$ vertices. Then
\begin{align*}
\varsigma_{\alpha}(G) \leq \sqrt{ nZ(G)},
\end{align*}
where $Z(G) = (1-\alpha)^2\|D(G)\|_F^2 + \alpha^2 \sum_{i=1}^n(Tr(v_i)-\frac{2W(G)}{n})^2$.
\end{thm}

\begin{proof}
Let
$\eta_i(G) = \sigma_i(G) - \frac{ 2\alpha W(G)}{n}$, where $i=1,\ldots,n$.
Clearly, $\sum_{i=1}^n\eta_i(G) = 0.$
From Equation (4), we have
\begin{align*}
&\sum_{i=1}^n(\eta_i)^2 \\
&=\sum_{i=1}^n(\sigma_i(G) - \frac{2\alpha W(G)}{n})^2\\
                        &= \sum_{i=1}^n (\sigma_i(G))^2 - \sum_{i=1}^n \frac{4\alpha \sigma_i(G) W(G)}{n} + \sum_{i=1}^n \frac{4\alpha^2W^2(G)}{n^2}\\
                        &= (1-\alpha)^2\|D(G)\|_F^2 + \alpha^2 \sum_{i=1}^nTr^2(v_i)- \frac{4\alpha W(\alpha) }{n}\sum_{i=1}^n  \sigma_i(G) + \sum_{i=1}^n \frac{4\alpha^2W^2(G)}{n^2}\\
                        &= (1-\alpha)^2\|D(G)\|_F^2 + \alpha^2 \sum_{i=1}^nTr^2(v_i)- \frac{4\alpha W(\alpha) }{n}\sum_{i=1}^n \alpha Tr(v_i) + \sum_{i=1}^n \frac{4\alpha^2W^2(G)}{n^2}\\
                        &= (1-\alpha)^2\|D(G)\|_F^2 + \alpha^2 \sum_{i=1}^n(Tr(v_i)-\frac{2W(G)}{n})^2.
\end{align*}
Let
$
Z(G)=(1-\alpha)^2\|D(G)\|_F^2 + \alpha^2 \sum_{i=1}^n(Tr(v_i)-\frac{2W(G)}{n})^2.
$
By Cauchy-Schwar$z'$s inequality, we have
\begin{align*}
(\varsigma_{\alpha}(G))^2 = ( \sum_{i=1}^n |\eta_i|)^2 \leq \sum_{i=1}^n (\eta_i(G))^2 \sum_{i=1}^n 1 =  nZ(G).
\end{align*}
\end{proof}

\begin{lem}
\textup{[3]}
Let $G$ be a graph with $n$ vertices. Then
\begin{eqnarray*}
Trace (D_{\alpha}(G))=\sum_{i=1}^n \sigma_i(G) = \alpha\sum_{i=1}^n Tr(v_i) =2\alpha W(G),\\
Trace (D_{\alpha}^2(G))=\sum_{i=1}^n \sigma_i^2(G) = \alpha^2 \sum_{i=1}^n Tr^2(v_i) + 2(1-\alpha)^2S,
\end{eqnarray*}
where $\sigma_1(G),\sigma_2(G),\ldots,\sigma_n(G)$ are the eigenvalues of $D_\alpha(G)$.
\end{lem}

\begin{thm}
Let $G$ be a graph with $n$ vertices. Then
\begin{eqnarray*}
\varsigma_{\alpha}(G) \leq \sqrt{(\alpha^2\sum_{i=1}^nTr^2(v_i) + 2(1-\alpha)^2S )n-4\alpha^2W^2(G)}.
\end{eqnarray*}

\end{thm}

\begin{proof}
From Cauchy-Schwar$z'$s inequality，we have
\begin{eqnarray*}
(\varsigma_{\alpha}(G))^2=(\sum_{i=1}^n |\sigma_i(G)-\frac{2\alpha W(G)}{n}|)^2 \leq \sum_{i=1}^n |\sigma_i(G)-\frac{2\alpha W(G)}{n}|^2 \sum_{i=1}^n 1.
\end{eqnarray*}%
It follows from Lemma 2.12 that

\begin{align*}
\varsigma_{\alpha}(G)     &\leq \sqrt{n\sum_{i=1}^n|\sigma_i(G)-\frac{2\alpha W(G)}{n}|^2}\\
                        &=  \sqrt{n\sum_{i=1}^n\sigma_i(G)^2- 4\alpha^2 W^2(G)}\\
                        &=  \sqrt{(\alpha^2\sum_{i=1}^nTr^2(v_i) + 2(1-\alpha)^2S )n-4\alpha^2W^2(G)}.
\end{align*}

\end{proof}

\section{Bounds for the $\alpha$-distance spectrum of graphs}
\begin{lem}
\textup{[8]} Let $x_1> x_2 \geq \ldots \geq x_n>0$ be $n$ real numbers. Then
\begin{eqnarray*}
\sum_{i=1}^n|x_i - M| < \frac{n}{2}x_1,
\end{eqnarray*}
where $M=\frac{\sum_{1=1}^n x_i}{n}$.
\end{lem}
It follows from the above lemma the following result holds directly.
\begin{prop}
For a graph $G$ with $n$ vertices, let $\sigma_1(G)$ be the largest eigenvalue of $D_\alpha(G)$. For $\alpha \in[\frac{1}{2},1)$,
\begin{eqnarray*}
n\sigma_1(G) > 2 \varsigma_{\alpha}(G).
\end{eqnarray*}
\end{prop}

For a matrix $M=(m_{ij})\in\mathbb{ C}^{n\times n}$, let $R_i(M)=\sum_{j=1}^n m_{ij}$ be the $i$-th row sum of $M$, $i=1,2,\cdots,n$.

\begin{lem}
\textup{[7,19]} Let $A=(a_{i,j})\in \mathbb{R}^{n\times n}$ be a nonnegative matrix. Let $\rho (A)$ be the spectral radius of $A$. Then
\begin{align*}
\min_{1\leq i\leq n}R_i(A)\leq \rho(A) \leq \max_{1\leq i\leq n}R_i(A).
\end{align*}
Further, if $A$ is an irreducible matrix, then the above two equalities hold if and only if $R_1(A)=R_2(A)=\cdots=R_n(A)$.
\end{lem}
From the above Lemma, it is easy to see the following result holds.
\begin{prop}
Let $G$ be a graph with $n$ vertices. Let $\rho (D_{\alpha}(G))$ be the spectral radius of $D_{\alpha}(G)$. Then

\begin{eqnarray*}
\min_{1\leq i\leq n} Tr(v_i)\leq \rho(D_{\alpha}(G)) \leq \max_{1\leq i\leq n}Tr(v_i),
\end{eqnarray*}
the equality holds  if and only if $G$ is a transmission regular graph.
\end{prop}

\begin{lem}
\textup{[22]}
For a graph $G$ with $n$ vertices, let $\lambda_1(G)$ be the largest eigenvalue of the signless Laplacian matrix $\widetilde{Q}(G)$. Let $p(x)$ be a polynomial on $x$. Then
\begin{eqnarray*}
\min_{1\leq i\leq n } R_i(p(\widetilde{Q}(G)))   \leq p(\lambda_1(G)) \leq \max_{1\leq i\leq n } R_i(p(\widetilde{Q}(G))).
\end{eqnarray*}
Moreover, if the row sums of $ p(\tilde{Q}(G))$ are not all equal, then both inequalities  are strict.
\end{lem}

Inspired by the above result, we give the following bounds on the largest eigenvalue of $D_{\alpha}(G)$.
\begin{prop}
For a graph $G$ with $n$ vertices, let $\sigma_1(G)$ be the largest eigenvalue of $D_\alpha(G)$.
Let $p(x)$ be a polynomial on $x$. For $\alpha \in[\frac{1}{2},1)$,
\begin{eqnarray*}
\min_{1\leq i\leq n}R_i(p(D_{\alpha}(G)))\leq p(\sigma_1(G)) \leq \max_{1\leq i\leq n}R_i(p(D_{\alpha}(G))).
\end{eqnarray*}
If the row sums of $p(D_{\alpha}(G))$ are not all equal, then both inequalities are strict.
\end{prop}

\begin{proof}
Perron-Frobenius Theorem gives that there exists a positive vector $\mathbf{x}=(x_1,\ldots,x_n)^{\rm T}$ such that $D_{\alpha}(G)\mathbf{x} = \sigma_1(G)\mathbf{x}$. Then
\begin{eqnarray*}
p(D_{\alpha}(G))\mathbf{x} = p(\sigma_1(G))\mathbf{x}.
\end{eqnarray*}
Let $\sum_{i=1}^n x_i =1$. Then

\begin{align*}
p(\sigma_1(G))&= p(\sigma_1(G)) \sum_{i=1}^n x_i = \sum_{i=1}^n p(\sigma_1(G))x_i = \sum_{i=1}^n(p(D_{\alpha}(G))\mathbf{x})_i\\
& = \sum_{i=1}^n x_i R_i(p(D_{\alpha}(G))).
\end{align*}
Hence,
\begin{eqnarray*}
\min_{1\leq i\leq n}R_i(p(D_{\alpha}(G)))\leq  \sum_{i=1}^n x_i R_i(p(D_{\alpha}(G))) \leq \max_{1\leq i\leq n}R_i(p(D_{\alpha}(G))).
\end{eqnarray*}
\end{proof}

\begin{lem}
For a graph $G$ with $n$ vertices, let $\mathbf{T}=\max\{Tr(v)|{v\in V(G)}\}$ and $\mathbf{t}=\min\{Tr(v)$ $|{v\in V(G)}\}$. Then for each vertex $u\in V(G)$,
$$2W(G) +(\mathbf{t}-1)Tr(u) - (n-1)\mathbf{t} \leq \sum_{v\neq u}d(u,v)Tr(v)  \leq2W(G) +(\mathbf{T}-1)Tr(u) - (n-1)\mathbf{T}.$$
\end{lem}
\begin{proof}
By calculation, we have
\begin{align*}
\sum_{v\neq u}d(u,v)Tr(v) &= \sum_{v\neq u}Tr(v) + \sum_{v\neq u}(d(u,v)-1)Tr(v)\\
                           &= 2W(G)- Tr(u) +  \sum_{v\neq u}(d(u,v)-1)Tr(v)\\
                          & \geq  2W(G)- Tr(u) +  \mathbf{t}\sum_{v\neq u}(d(u,v)-1)\\
                          & = 2W(G)- Tr(u) +  \mathbf{t}(\sum_{v\neq u}d(u,v)- (n-1))\\
                          & = 2W(G)- Tr(u) +  \mathbf{t}(Tr(u) - (n-1))\\
                          & = 2W(G) +(\mathbf{t}-1)Tr(u) - (n-1)\mathbf{t}.
\end{align*}
Similarly,
\begin{eqnarray*}
 \sum_{v\neq u}d(u,v)Tr(v)  \leq2W(G) +(\mathbf{T}-1)Tr(u) - (n-1)\mathbf{T}.
\end{eqnarray*}
\end{proof}

\begin{thm}
For a graph $G$ with $n$ vertices, let $\mathbf{T}=\max\{Tr(v)|{v\in V(G)}\}$ and $\mathbf{t}=\min\{Tr(v)$ $|{v\in V(G)}\}$. Let $\sigma_1(G)$ be the largest eigenvalue of $D_\alpha(G)$, where $\alpha \in [\frac{1}{2},1)$. Then for $u\in V(G)$,
\begin{align*}
&\frac{(1-\alpha)(\mathbf{t}-1) +\sqrt{(1-\alpha)^2(\mathbf{t}-1)^2 - 4(\alpha \mathbf{t}^2 + 2(1-\alpha)W(G)- (1-\alpha)(n-1)\mathbf{t} }  }{2} \\
&\leq \sigma_1(G)\leq\\
&\frac{(1-\alpha)(\mathbf{T}-1) +\sqrt{(1-\alpha)^2(\mathbf{T}-1)^2 - 4(\alpha \mathbf{T}^2 + 2(1-\alpha)W(G)- (1-\alpha)(n-1)\mathbf{T} }  }{2}.
\end{align*}
\end{thm}

\begin{proof}
Let $\mathbf{e}_i\in\mathbb{R}^{n}$ be an unity vector whose $i$-component is $1$. And let $\mathbf{1}\in\mathbb{R}^{n}$ be a vector whose components are all $1$. Then
\begin{align*}
&R_i(D_{\alpha}^2(G))\\
  &=\alpha^2 \mathbf{e}_i^{\rm T} Tr^2(G)\mathbf{1} +(1-\alpha)^2 \mathbf{e}_i^{\rm T} D^2(G) \mathbf{1}+\alpha(1-\alpha)\mathbf{e}_i^{\rm T} Tr(G) D(G) \mathbf{1}\\
  &+ \alpha(1-\alpha)\mathbf{e}_i^{\rm T}  D(G) Tr(G) \mathbf{1}\\
&=\alpha Tr^2(v_i) + (1-\alpha)\sum_{v\neq v_i}d(v_i,v)Tr(v).
\end{align*}
From Lemma 3.7, we have
\begin{align*}
R_i(D_{\alpha}^2(G)) \geq \alpha Tr^2(v_i) + (1-\alpha)(2W(G) +(\mathbf{t}-1)Tr(v_i) - (n-1)\mathbf{t}),
\end{align*}
and
\begin{align*}
R_i(D_{\alpha}^2(G)) \leq \alpha Tr^2(v_i) + (1-\alpha)(2W(G) +(\mathbf{T}-1)Tr(v_i) - (n-1)\mathbf{T}).
\end{align*}
Let $p(x) = x^2 - (1-\alpha)(\mathbf{t}-1)x $. Then the sum of the entries in the $i$-th row of $p(D_{\alpha}(G))$ is
\begin{align*}
R_i(p(D_{\alpha}(G)) &= R_i(D_{\alpha}^2(G) - (1-\alpha)(\mathbf{t}-1)D_{\alpha}(G))\\
                    &=R_i(D_{\alpha}^2(G)) - (1-\alpha)(\mathbf{t}-1)R_i(D_{\alpha}(G))\\
                    &=R_i(D_{\alpha}^2(G)) - (1-\alpha)(\mathbf{t}-1)Tr(v_i).
\end{align*}
Clearly,
\begin{align*}
R_i(p(D_{\alpha}(G))) &\geq  \alpha Tr^2(v_i) + (1-\alpha)(2W(G) +(\mathbf{t}-1)Tr(v_i) - (n-1)\mathbf{t}) \\
                    & -(1-\alpha)(\mathbf{t}-1)Tr(v_i)\\
                    &=\alpha Tr^2(v_i) +2(1-\alpha)W(G) - (1-\alpha)(n-1)\mathbf{t} \\
                    & \geq \alpha \mathbf{t}^2 + 2(1-\alpha)W(G)- (1-\alpha)(n-1)\mathbf{t}.
\end{align*}
By Proposition 3.6, we obtain
\begin{eqnarray*}
\alpha \mathbf{t}^2 + 2(1-\alpha)W(G)- (1-\alpha)(n-1)\mathbf{t} \leq p(\sigma_1(G)) = (\sigma_1(G))^2-(1-\alpha)(\mathbf{t}-1)\sigma_1(G),
\end{eqnarray*}
that is
\begin{eqnarray*}
\sigma_1(G)\geq \frac{(1-\alpha)(\mathbf{t}-1) +\sqrt{(1-\alpha)^2(\mathbf{t}-1)^2 - 4(\alpha \mathbf{t}^2 + 2(1-\alpha)W(G)- (1-\alpha)(n-1)\mathbf{t} }  }{2}.
\end{eqnarray*}
Similarly, let $p(x) = x^2-(1-\alpha)(\mathbf{T}-1)x$. Then
$$\sigma_1(G)\leq \frac{(1-\alpha)(\mathbf{T}-1) +\sqrt{(1-\alpha)^2(\mathbf{T}-1)^2 - 4(\alpha \mathbf{T}^2 + 2(1-\alpha)W(G)- (1-\alpha)(n-1)\mathbf{T} }  }{2}.$$
%

\end{proof}

\begin{lem}\textup{[26]} 
let $x_1\geq x_2\geq ,\ldots,\geq x_m$ be real numbers such that $\sum_{i=1}^m x_i =0$. Then
\begin{eqnarray*}
x_1 \leq \sqrt{\frac{m-1}{m}\sum_{i=1}^m x_i^2},
\end{eqnarray*}
the equality holds if and only if $x_1=\cdots=x_m=-\frac{x_1}{m-1}$.
\end{lem}

\begin{thm}
Let $G$ be a graph with $n$ vertices. Then

\begin{eqnarray*}
\sigma_1(G) \leq \frac{2\alpha W(G)}{n} +\sqrt{\frac{n-1}{n} (\|D_{\alpha(G)}\|_F^2 - \frac{4\alpha^2W^2(G)}{n})},
\end{eqnarray*}
the equality holds if and only if $G\cong K_n$.
\end{thm}
\begin{proof}
Obviously,
\begin{eqnarray*}
\sum_{i=1}^n(\sigma_i(G)- \frac{2\alpha W(G)}{n}) =0.
\end{eqnarray*}
It follows from Lemma 3.9 that
\begin{eqnarray*}
\sigma_1(G)- \frac{2\alpha W(G)}{n} \leq \sqrt{\frac{n-1}{n} \sum_{i=1}^n(\sigma_i(G)-\frac{2\alpha W(G)}{n})^2  },
\end{eqnarray*}
and the equality holds if and only if
\begin{eqnarray*}
\sigma_2(G) - \frac{2\alpha W(G)}{n}=\cdots = \sigma_n(G) - \frac{2\alpha W(G)}{n} = -\frac{\sigma_1-\frac{2\alpha W(G)}{n}}{n-1}.
\end{eqnarray*}
By calculation, we know

\begin{align*}
\sum_{i=1}^n( \sigma_i(G) - \frac{2\alpha W(G)}{n})^2 &= \sum_{i=1}^n (\sigma_i(G))^2 - 2\frac{2\alpha W(G)}{n} \sum_{i=1}^n \sigma_i(G) + n \frac{4\alpha^2W^2(G)}{n^2}\\
                                                      &=\|D_{\alpha}(G)\|_F^2 - \frac{8\alpha^2W^2(G)}{n} +\frac{4\alpha^2W^2(G)}{n}\\
                                                      &= \|D_{\alpha}(G)\|_F^2 - \frac{4\alpha^2W^2(G)}{n}.
\end{align*}

Then
\begin{eqnarray*}
\sigma_1(G) \leq \frac{2\alpha W(G)}{n} +\sqrt{\frac{n-1}{n} (\|D_{\alpha(G)}\|_F^2 - \frac{4\alpha^2W^2(G)}{n})},
\end{eqnarray*}
and the equality holds if and only if $\sigma_2(G) - \frac{2\alpha W(G)}{n}=\cdots = \sigma_n(G) - \frac{2\alpha W(G)}{n} = -\frac{\sigma_1(G)-\frac{2\alpha W(G)}{n}}{n-1}$. And it follows from Proposition 2.3 that the equality holds if and only if $G\cong K_n$.
%
\end{proof}

For an $n \times n$ matrix $M$ and order partition $(K_1,K_2,\ldots,K_m)$ of the ordered set $\{ 1,2,\ldots,n \}$, $M$ can be denoted as a partition matrix
\begin{eqnarray*}
M= \begin{pmatrix}M_{11} & M_{12} & \cdots & M_{1m} \\
          M_{21} & M_{22} & \cdots & M_{2m} \\
          \ldots & \cdots & \ddots &\cdots \\
           M_{m1} & M_{m2} & \cdots & M_{mm} \\
                                 \end{pmatrix},
\end{eqnarray*}
where $M_{ij}$ has $K_i$ as the set of its row indices and $K_j$ as the set of its column indices, Let $B_M$ be the quotient matrix of the partitioned matrix $M$, which is defined to be an $ m \times m$ matrix with the $ij$-entry $(B_M)_{ij}=\frac{1}{|K_i|}\sum_{l=1}^{K_i}\sum_{k=1}^{K_j}{(M_{ij})_{lk}}$, where $1\leq i, j\leq m$ (see \textup{[5]} ).

\begin{lem}
\textup{[16]}
Suppose $B_M\in \mathbb{C}^{m\times m}$ is a quotient matrix of a symmetric partitioned matrix $M\in \mathbb{C}^{n\times n}$. Let $\{\mu_1,\cdots,\mu_m\}$ and $\{\lambda_1,\cdots,\lambda_n\}$ be the eigenvalue sets of $B_M$ and $M$, respectively. Then for $i=1,2,\cdots, m$,
$$\lambda_i\geqslant \mu_i \geqslant \lambda_{n-m+i}.$$
\end{lem}

\begin{prop}
Let $G$ be a graph with $n$ vertices. Let $\sigma_1(G)$ and $\sigma_n(G)$ be the largest and the smallest eigenvalue of $D_\alpha(G)$, respectively. Then for each $v_i\in V(G)$,
\begin{align*}
&\sigma_1(G) -\sigma_n(G) \geq \\
&\frac{ \sqrt{(2W(G)-2Tr(v_i)+ \alpha n Tr^2(v_i))^2- 4(n-1)(2\alpha Tr(v_i)W(G) - Tr^2(v_i))}}{n-1}.
\end{align*}
\end{prop}

\begin{proof}
Let $K_1=\{i\}$ and $K_2=\{V(G)\setminus i\}$ be a partition of $V(G)$. Then the following $B_M$ is the corresponding quotient matrix of $D_{\alpha}(G)$,
\begin{eqnarray*}
B_M = \begin{pmatrix}\alpha Tr(v_i) & (1-\alpha ) Tr(v_i)  \\
          \frac{(1-\alpha)Tr(v_i)}{n-1} & \frac{2W(G) - (2-\alpha)Tr(v_i)}{n-1} \end{pmatrix}.
\end{eqnarray*}
Let $\{\sigma_1(G), \cdots, \sigma_n(G)\}$ and $\{\lambda_1 (B_M), \lambda_1 (B_M)\}$ be the eigenvalue sets of $D_\alpha(G)$ and $B_M$, respectively.
It follows from Lemma 3.11 that
\begin{eqnarray*}
\sigma_i(G)\geq \lambda_i (B_M) \geq \sigma_{n-2+i}(G),
\end{eqnarray*}
where $i=1,2$. Hence,

$\sigma_1(G)\geq \lambda_1 (B_M)$ and $\lambda_2(B_M)\geq \sigma_n(G)$.

Then
\begin{eqnarray*}
\sigma_1(G)-\sigma_n(G) \geq \lambda_1(B_M) - \lambda_2(B_M).
\end{eqnarray*}
Since $\det(xI-M_B)=0$, we obtain
\begin{align*}
\lambda_1(B_M)=&\frac{2W(G) + (\alpha n -2)Tr(v_i)}{2(n-1)}\\
 &+ \frac{ \sqrt{(2W(G)-2Tr(v_i)+ \alpha n Tr^2(v_i))^2- 4(n-1)(2\alpha Tr(v_i)W(G) - Tr^2(v_i))}}{2(n-1)}
\end{align*}

\begin{align*}
\lambda_2(B_M)= &\frac{2W(G) + (\alpha n -2)Tr(v_i)}{2(n-1)} \\
           &-\frac{\sqrt{(2W(G)-2Tr(v_i)+ \alpha n Tr^2(v_i))^2- 4(n-1)(2\alpha Tr(v_i)W(G) - Tr^2(v_i))}}{2(n-1)}
\end{align*}

\begin{align*}
&\lambda_1(B_M)-\lambda_2(B_M)\\ 
&= \frac{ \sqrt{(2W(G)-2Tr(v_i)+ \alpha n Tr^2(v_i))^2- 4(n-1)(2\alpha Tr(v_i)W(G) - Tr^2(v_i))}}{n-1}.
\end{align*}
%
\end{proof}

 Let $G$ be a graph with $n$ vertices. Let $\{\lambda_1(G),\lambda_2(G),\cdots,\lambda_n(G)\}$ be the eigenvalue set of $\widetilde{A}(G)$. And $EE(G)=\sum_{i=1}^n e^{\lambda_i(G)}=\sum_{i=1}^n\sum_{k=0}^\infty \frac{\lambda_i(G)^k}{k!}$ is called the Estrada index of $G$. It is well-known that
Estrada index plays an important role in the problem of characterizing the molecular structure [30]. In [15], the study was extended to distance matrices, and the distance Estrada index of $G$ was defined by $DEE(G)=\sum_{i=1}^n e^{o_i(G)}=\sum_{i=1}^n\sum_{k=0}^\infty \frac{o_i(G)^k}{k!}$, where $\{o_1(G),o_2(G),\ldots,o_n(G)\}$ is the eigenvalue set of $D(G)$.

In this paper, we consider a more general Estrada index.
Let $$DEE_{\alpha}(G)=\sum_{i=1}^n e^{\sigma_i(G)}=\sum_{i=1}^n\sum_{k=0}^\infty \frac{\sigma_i(G)^k}{k!}$$ be the \textit{$\alpha$-distance Estrada index} of $G$, where $\{\sigma_1(G),\cdots,\sigma_n(G)\}$ is the eigenvalue set of $D_\alpha(G)$. Clearly, $DEE_{0}(G)=DEE(G)$. Next, we establish some bounds on the $\alpha$-distance Estrada index.

\begin{lem}
\textup{[25]}
Let $x_1,x_2,\ldots,x_n$ be nonnegative real numbers. Then for $k\geq 2$,
\begin{eqnarray*}
\sum _{i=1}^n x_i^k \leq \left(\sum _{i=1}^n x_i^2 \right)^{k/2}.
\end{eqnarray*}
\end{lem}

\begin{thm}
Let $G$ be a graph with $n$ vertices. Then
\begin{eqnarray*}
DEE_\alpha(G)\leq n + 2\alpha W(G) - 1 -\omega  + e^{\omega},
\end{eqnarray*}
where $\omega = \sqrt{\alpha^2 \sum_{i=1}^n Tr(v_i)^2 + 2(1-\alpha)^2S} $.
\end{thm}

\begin{proof}
Let $\sigma_1(G),\ldots,\sigma_n(G)$ be the eigenvalues of $D_{\alpha}(G)$. From Lemmas 2.12 and 3.13, we have

\begin{align*}
DEE_\alpha(G) & = n + 2\alpha W(G) + \sum_{i=1}^n\sum_{k\geq 2} \frac{\sigma_i(G)^k}{k!} \\
             & \leq n + 2\alpha W(G)  + \sum_{i=1}^n\sum_{k\geq 2} \frac{|\sigma_i(G)|^k }{k!}\\
               & = n + 2\alpha W(G) + \sum_{k\geq 2} \frac{1}{k!}\sum_ {i=1}^n  |\sigma_i(G)| ^{k}\\
               & \leq n + 2\alpha W(G) + \sum_{k\geq 2}\frac{1}{k!}\left( \sum_{i=1}^n \sigma_i(G)^2 \right)^{\frac{k}{2}}\\
               & = n +  2\alpha W(G) + \sum_{k\geq 2}\frac{1}{k!} \left(\alpha^2 \sum_{i=1}^n Tr(v_i)^2 + 2(1-\alpha)^2S \right)^{\frac{k}{2}}&&&&&&&&&&&
\end{align*} 
\begin{align*}              
               & =  n + 2\alpha W(G) - 1 -\omega  + \sum_{k\geq 0}\frac{1}{k!} \omega ^k\\
               & = n + 2\alpha W(G) - 1 -\omega  + e^{\omega},&&&&&&&&&&&&&&&&
\end{align*}
where $\omega = \sqrt{\alpha^2 \sum_{i=1}^n Tr(v_i)^2 + 2(1-\alpha)^2S} $.
\end{proof}

\begin{thm}
Let $G$ be a graph with $n$ vertices. Then
\begin{eqnarray*}
DEE_{\alpha}(G) \geq \sqrt{n + 4 \alpha W(G) + n(n-1) e^{\frac{4\alpha W(G)}{n}}}.
\end{eqnarray*}

\end{thm}
\begin{proof}
Let $\sigma_1(G),\ldots,\sigma_n(G)$ be the eigenvalues of $D_{\alpha}(G)$. Then
\begin{eqnarray*}
(DEE_{\alpha}(G))^2 & = \sum_{i=1}^n e^{2\sigma_i(G)} + 2\sum_{1\leq i<j\leq n} e^{\sigma_i(G)} e^{\sigma_j(G)}.
\end{eqnarray*}
From Arithmetic-Geometric inequality, we obtain

\begin{align*}
2\sum_{1\leq i<j\leq n} e^{\sigma_i(G)} e^{\sigma_j(G)} &\geq n(n-1) \left(\prod_{1\leq i<j\leq n}e^{\sigma_i(G)} e^{\sigma_j(G)} \right)^{\frac{2}{n(n-1)}}\\
                                      &= n(n-1) \left( \left( \prod_{i=1}^n e^{\sigma _i(G)} \right)^{n-1} \right)^{\frac{2}{n(n-1)}}\\
                                      &= n(n-1) \left( e^{\sum_i ^n \sigma_i(G)  } \right)^{\frac{2}{n}}\\
                                      &= n(n-1) e^{\frac{4\alpha W(G)}{n}}.
\end{align*}

By Taylor expansion theorem, we have
\begin{align*}
\sum_{i=1}^n e^{2\sigma_i(G)} &=  \sum_{i=1}^n \sum_{k\geq 0} \frac{(2\sigma_i(G))^k}{k!} \\
                           &= n +4\alpha W(G) + \sum_{i=1}^{n} \sum_{k\geq 2} \frac{(2\sigma_i(G))^k}{k!}.
\end{align*}

Let $\delta\in[0,4]$. Then
\begin{align*}
\sum_{i=1}^n e^{2\sigma_i(G)} &\geq  n +4\alpha W(G) + \delta\sum_{i=1}^{n} \sum_{k\geq 2} \frac{\sigma_i(G)^k}{k!}  \\
                           &=  n +4\alpha W(G) - \delta n- 2\alpha \delta W(G) + \delta \sum_{i=1}^{n}\sum_{k\geq 0} \frac{\sigma_i(G)^k}{k!}\\
                           &= (1-\delta)n + (4 \alpha -2\delta\alpha)W(G) + \delta DEE_\alpha(G).
\end{align*}
Hence,
\begin{eqnarray*}
(DEE_\alpha(G))^2 \geq (1-\delta)n + (4 \alpha -2\delta\alpha)W(G) + \delta DEE_{\alpha}(G) + n(n-1) e^{\frac{4\alpha W(G)}{n}},\\
DEE_\alpha(G)  \geq \frac{\delta}{2}  + \sqrt{(1-\delta)n + (4 \alpha -2\delta\alpha)W(G) + \frac{\delta^2}{4} + n(n-1) e^{\frac{4\alpha W(G)}{n}}}.\\
\end{eqnarray*}
It is elementary to show for $n\geq 2$, let the function
\begin{eqnarray*}
f(x)=\frac{x}{2}  + \sqrt{(1-x)n + (4 \alpha -2\alpha x)W(G) + \frac{x^2}{4} + n(n-1) e^{\frac{4\alpha W(G)}{n}}},\\
\end{eqnarray*}
where $0\leq\alpha\leq1$, then $f(x)$ is monotonically decreasing in  the interval $[0,4]$. Hence, $f(0)$ is a largest lower bound of $DEE_\alpha(G)$.
\end{proof}

\begin{thm}
Let $G$ be a graph with $n$ vertices. Then
\begin{eqnarray*}
DEE_{\alpha}(G) \geq e^{\frac{2W(G)}{n}}+ (n-1) + 2 \alpha W(G) - \frac{2W(G)}{n}
\end{eqnarray*}
\end{thm}
\begin{proof}
Let $f(x) =(x-1)-\ln x$, where $x>0.$
Obviously, $f(x)$ is a decreasing function when $x \in (0,1]$, and $f(x)$ is increasing when $x \in [1,+\infty)$. Then $f(x)\geq f(1)=0$, that is
\begin{eqnarray*}
x\geq 1+ \ln x, ~x>0,
\end{eqnarray*}
and the equality holds if and only if $x=1$. So by this function, we have

\begin{align*}
DEE_{\alpha}(G) &\geq e^{\sigma_1(G)} + (n-1) + \sum_{k\geq2} \ln e^{\sigma_k(G)}\\
                &= e^{\sigma_1(G)}+ (n-1) + \sum_{k\geq2} {\sigma_k(G)}\\
                &= e^{\sigma_1(G)}+ (n-1) + 2 \alpha W(G) - \sigma_1(G),
\end{align*}

where $\sigma_1(G)> \sigma_2(G),\ldots,\geq\sigma_k(G)$ are the eigenvalues of $D_\alpha(G)$.

Let $\Gamma(x) = e^x + (n-1) + 2 \alpha W(G) - x$, where $x>0$.
Clearly, $\Gamma(x)$ is an increasing function when $x \in (0, +\infty)$.
From Lemma \ref{lem1}, we have
\begin{eqnarray*}
\sigma_1(G) \geq \frac{2W(G)}{n}\geq0,
\end{eqnarray*}
then,
\begin{eqnarray*}
\begin{aligned}
\Gamma(\sigma_1(G)) \geq \Gamma(\frac{2W(G)}{n}).
\end{aligned}
\end{eqnarray*}
Hence,
\begin{eqnarray*}
DEE_{\alpha}(G) \geq e^{\frac{2W(G)}{n}}+ (n-1) + 2 \alpha W(G) - \frac{2W(G)}{n}.
\end{eqnarray*}
\end{proof}

From Theorem 2.1, we have the following result.
\begin{cor}
Let $G$ be a transmission regular graph with $n$ vertices. Let $Tr(u)=r$ for each $u\in V(G)$. Then
\begin{eqnarray*}
DEE_{\alpha}(G) \geq e^{r}+ (n-1) + 2 \alpha r  - r.
\end{eqnarray*}
\end{cor}

Next, a new upper bound for the $\alpha$-distance Estrada index is established.
\begin{thm}
Let $G$ be a graph with $n$ vertices. Then
\begin{eqnarray*}
DEE_\alpha(G)\leq  e^{\frac{2\alpha W(G)}{n}}(n-1- \varsigma_\alpha(G) +e^{\varsigma_\alpha(G)}).
\end{eqnarray*}
\end{thm}
\begin{proof}
By the definition of $\alpha$-distance energy, we have

\begin{align*}
 DEE_\alpha(G) = e^{\frac{2\alpha W(G)}{n}}(\sum_{i=1}^n e^{\sigma_i(G)-\frac{2\alpha W(G)}{n}}) &= e^{\frac{2\alpha W(G)}{n}}( n +  \sum_{i=1}^n\sum_{k\geq 2} \frac{(\sigma_i-\frac{2\alpha W(G)}{n})^k}{k!} )\\
                                                          &\leq e^{\frac{2\alpha W(G)}{n}}( n +  \sum_{i=1}^n\sum_{k\geq 2} \frac{|\sigma_i-\frac{2\alpha W(G)}{n}|^k}{k!})\\
                                                          &= e^{\frac{2\alpha W(G)}{n}}(n+ \sum_{k\geq 2} \frac{1}{k!}   \varsigma_\alpha(G)^k)\\
                                                          &= e^{\frac{2\alpha W(G)}{n}}(n-1- \varsigma_\alpha(G) +e^{\varsigma_\alpha(G)}) .
\end{align*}

\end{proof}

\section{Acknowledgments}
Supported by  by the National Natural Science Foundation of China (No. 11801115 and No.11601102), the Natural Science Foundation of the Heilongjiang Province (No.QC2018002) and the Fundamental Research Funds for the Central Universities.


\vspace{3mm}

\noindent
\textbf{References}
\begin{thebibliography}{}


\bibitem{}M. Aouchiche, P. Hansen, Two Laplacians for the distance matrix of a graph. Linear Algebra Appl. 439(1)(2013) 21-33.
\bibitem{}R.A. Brualdi, Energy of a graph, Notes for AIM Workshop On Spectra of families of matrices described by graphs, digraphs and sign patterns. 2006.
\bibitem{}S.Y. Cui, J.X. He, G.X. Tian, The generalized distance matrix. Linear Algebra Appl. 563(2019) 1-23.
\bibitem{}S.Y. Cui, G.X. Tian, L. Zheng, On the generalized distance spectral radius of graphs. arXiv:1901.07695, 2019.
\bibitem{}D. Cvetkovi$\acute{c}$, P. Rowlinson, S. Simi$\acute{c}$. An introduction to the theory of Graph Spectra. Cambridge University Press, 2010.
\bibitem{}K.C. Das, M. Aouchiche, P. Hansen, On (distance) Laplacian energy and (distance) signless Laplacian energy of graphs. Discrete Appl. Math. 243(2018) 172-185.
\bibitem{}R.C. Díaz, A.I. Julio, O. Rojo, New bounds on the distance Laplacian and distance signless Laplacian spectral radii. arXiv:1804.06335, 2018.
\bibitem{}R.C. Díaz, O. Rojo, Sharp upper bounds on the distance energies of a graph. Linear Algebra Appl. 54(2018) 55-75.
\bibitem{}H. Dong, X. Guo, Ordering trees by their wiener indices. MATCH Commun. Math. Comput. Chem. 25(2006) 527-540.
\bibitem{}R.L. Graham, L. Lovász, Distance matrix polynomials of trees. Adv. Math. 29(1)(1978) 60-88.
\bibitem{}R.L. Graham, H.O. Pollak. On the Addressing Problem for Loop Switching. Bell Syst. Tech. J. 50(8)(1971) 2495-2519.
\bibitem{}I. Gutman, Acyclic systems with extremal H$\ddot{u}$ckel $\pi$-electron energy. Theoret. Chim. Acta. 45(2)(1977) 79-87.
\bibitem{}I. Gutman, The energy of a graph. Ber. Math-Statist. Sekt. Forschungsz. Graz. 103(1978) 1-22.
\bibitem{}H.Y. Guo, B. Zhou, On the distance $\alpha $-spectral radius of a connected graph. arXiv:1901.10180, 2019.
\bibitem{}A.D. Güng\"{o}r, S.B. Bozkurt, On the distance Estrada index of graphs. Hacet. J. Math. Stat. 38(3)(2009) 277-283.
\bibitem{}W.H. Haemers, Interlacing eigenvalues and graphs. Linear Algebra Appl. 226(1995) 593-616.
\bibitem{}G. Indulal, I. Gutman, On the distance spectra of some graphs. Math. Commun. 13(1)(2008) 123-131.
\bibitem{}J.H. Koolen, V. Moulton, Maximal energy graphs. Adv. Appl. Math. 26(1)(2001) 47-52.
\bibitem{}J.S. Li, Y.L. Pan, Upper bounds for the Laplacian graph eigenvalues, Acta Math. Sin. (Engl. Ser.), 20(5)(2004) 803-806.
\bibitem{}H. Lin, J. Xue, J. Shu, On the $D_{\alpha}$-spectra of graphs. Linear Multilinear Algebra. doi: 10.1080/03081087.2019.1618236.
\bibitem{}H. Lin, B. Zhou, The effect of graft transformations on distance signless Laplacian spectral radius. Linear Algebra Appl. 504(2016) 433-461.
\bibitem{}H.Q. Liu, M. Lu, F. Tian, On the Laplacian spectral radius of a graph. Linear Algebra Appl. 376(2004) 135-141.
\bibitem{}B.J. McClelland, Properties of the latent roots of a matrix: The estimation of $\pi$-electron energies. J. Chem. Phys. 54(2)(1971) 640-643.
\bibitem{}V. Nikiforov, Merging the A- and Q-spectral theories. Appl. Anal. Discrete Math. 11(1)(2016) 81-107.
\bibitem{}N.J. Rad, A. Jahanbani, I. Gutman, Zagreb energy and Zagreb Estrada index of graphs. MATCH Commun. Math. Comput. Chem. 79(2018) 371-386.
\bibitem{}O. Rojo, H. Rojo, A decreasing sequence of upper bounds on the largest Laplacian eigenvalue of a graph. Linear Algebra Appl.  381(2004) 97-116.
\bibitem{}P.M. Winkler, Isometric embedding in products of complete graphs. Discrete Appl. Math. 7(2)(1984) 221-225.
\bibitem{}R. Xing, B. Zhou, J. Li, On the distance signless Laplacian spectral radius of graphs. Linear Multilinear Algebra. 62(10)(2014) 1377-1387.
\bibitem{}R. Xing, B. Zhou, On the distance and distance signless Laplacian spectral radii of bicyclic graphs. Linear Algebra Appl. 439(12)(2013) 3955-3963.
\bibitem{}B. Zhou, I. Gutman, More on the Laplacian Estrada index. Appl. Anal. Discrete Math. 3 (2009) 371-378.


\end{thebibliography}
\end{document}